\documentclass[12pt,a4paper]{article}
\usepackage[small,sc]{titlesec}
\usepackage{amssymb}
\usepackage{amsthm}
\usepackage{graphicx}
\usepackage{amsmath}
\theoremstyle{plain}
\newtheorem{thm}{\sc \bf{Theorem}}[section]
\newtheorem{lem}[thm]{\sc \bf{Lemma}}
\newtheorem{remark}[thm]{\sc \bf{Remark}}

\title{\bf{On Intersection Representations and Clique Partitions of Graphs}\footnote{
Work supported partly by the National Science Council of Taiwan,
Republic of China, under the Grant No. NSC 96-2115-M-029-001}}
\author{Tao-Ming Wang\footnote{The corresponding author with e-mail wang@thu.edu.tw}~ and Jun-Lin Kuo \\\\
Institute of Applied Mathematics \\and\\
Department of Mathematics \\
Tunghai University\\
Taichung, Taiwan 40704 }
\date{\today}
\begin{document}
\maketitle

\begin{abstract}
A multifamily set representation of a finite simple graph $G$ is a
multifamily $\mathcal{F}$ of sets (not necessarily distinct) for
which each set represents a vertex in $G$ and two sets in
$\mathcal{F}$ intersects if and only if the two corresponding
vertices are adjacent. For a graph $G$, an \textit{edge clique
covering} (\textit{edge clique partition}, respectively)
$\mathcal{Q}$ is a set of cliques for which every edge is contained
in \textit{at least} (\textit{exactly}, respectively) one member of
$\mathcal{Q}$. In 1966, P. Erd\"{o}s, A. Goodman, and L. P\'{o}sa
(The representation of a graph by set intersections,
\textit{Canadian J. Math.}, \textbf{18}, pp.106-112) pointed out
that for a graph there is a one-to-one correspondence between
multifamily set representations $\mathcal{F}$ and clique coverings
$\mathcal{Q}$ for the edge set. Furthermore, for a graph one may
similarly have a one-to-one correspondence between particular
multifamily set representations with intersection size at most one
and clique partitions of the edge set. In 1990, S. McGuinness and R.
Rees (On the number of distinct minimal clique partitions and clique
covers of a line graph, \textit{Discrete Math.} \textbf{83} (1990)
49-62.) calculated the number of distinct clique partitions for line
graphs. In this paper, we study the set representations of graphs
corresponding to edge clique partitions in various senses, namely
family representations of \textit{distinct} sets, antichain
representations of \textit{mutually exclusive} sets, and uniform
representations of sets with the \textit{same cardinality}. Among
others, we completely determine the number of distinct family
representations and the number of antichain representations of line
graphs.
\end{abstract}

\section{\bf{Background and Introduction}}

By a \textit{multigraph} $M=(V(M),E(M),q)$ we mean a triple
consisting of a set $V(M)$ of \textit{vertices}, a set $E(M)$ of
\textit{edges}, and a function $q$ defined in the following manner.
For each unordered pair $\{u,v\} \subset V(M)$, let $q(u,v)$ be the
number of \textit{parallel edges} joining $u$ with $v$, and we call
it the \textit{multiplicity}. If $q(u,v)\neq 0$, then we say that
$\{u,v\}$ is an \textit{edge} of $M$. In this paper unless otherwise
stated we consider only \textit{finite, undirected, simple} graphs.
That is, $q(u,v)\leq 1$ for every $\{u,v\}\subset V(M)$ and
$q(u,u)=0$ for every $u\in V(M)$. We simply call them graphs instead
of multigraphs throughout this article.
For a subset $S\subseteq V(M)$, $\langle S\rangle_V$ denotes the
\textit{subgraph induced by} $S$. For a vertex $v$ in a graph $G$,
let $d_G(v)$ or $d(v)$ denote the \textit{degree} of $v$ in $G$
which is the number of neighbors of $v$ in $G$. For other
terminology we do not define here please refer to \cite{west}.

Let $\mathcal{F}=\{S_1,...,S_p\}$ be a multifamily of nonempty
subsets of a finite non-empty set $X$, where $S_1,...,S_p$ might not
be distinct. $\textbf{S}(\mathcal{F})$ denotes the union of sets in
$\mathcal{F}$. The \textit{intersection multigraph} of
$\mathcal{F}$, denoted $\Omega(\mathcal{F})$, is defined by
$V(\Omega(\mathcal{F}))=\mathcal{F}$, with $|S_i\cap
S_j|=q(S_i,S_j)$ whenever $i\neq j$.

We say that a multigraph $M$ is (isomorphic to) an intersection
multigraph on $\mathcal{F}$ if there exists a multifamily
$\mathcal{F}$ such that $M\cong \Omega(\mathcal{F})$. In this case,
we also say that $\mathcal{F}$ is a \textit{multifamily
representation} of the multigraph $M$. It is easy to see that every
multigraph is isomorphic to an intersection multigraph on some
multifamily, therefore one may define the \textit{multifamily
intersection number}, denoted $\omega_{{m}}(M)$, of a given
multigraph $M$ to be the minimum cardinality of a set $X$ such that
$M$ is isomorphic to an intersection multigraph on a multifamily
$\mathcal{F}$ of subsets of $X$. In this case we also say that
$\mathcal{F}$ is a \textit{minimum multifamily representation} of
$M$.

If a multigraph $M$ is isomorphic to an intersection multigraph
$\Omega(\mathcal{F})$, then we have the following additional
conventions:
\begin{enumerate}
\item $\mathcal{F}$ is called a {\bf family representation} of the multigraph $M$
if $\mathcal{F}$ is a family of {\bf distinct} subsets of $X$;
\item $\mathcal{F}$ is called an {\bf antichain representation} of $M$ if
$\mathcal{F}$ is an antichain with respect to set inclusions, that
is, any two sets in $\mathcal{F}$ are mutually exclusive.
\item $\mathcal{F}$ is called an {\bf uniform representation} of $M$ if
$\mathcal{F}$ is an uniform family of distinct sets with the same
cardinality.
\end{enumerate}
It is not hard to see that every multigraph is isomorphic to an
intersection multigraph $\Omega(\mathcal{F})$, where $\mathcal{F}$
can be required to be either a family representation, an antichain
representation, or an uniform representation. Therefore it makes
sense to define the following notions.
\begin{enumerate}
\item The {\bf family intersection number} of $M$, denoted $\omega_f(M)$, is
the minimum cardinality of $\textbf{S}(\mathcal{F})$ for which $M$
has a family representation $\mathcal{F}$.
\item The {\bf antichain intersection number} of $M$, denoted
$\omega_a(M)$, is the minimum cardinality of
$\textbf{S}(\mathcal{F})$ for which $M$ has an antichain
representation $\mathcal{F}$.
\item The {\bf uniform intersection number} of $M$, denoted
$\omega_u(M)$, is the minimum cardinality of
$\textbf{S}(\mathcal{F})$ for which $M$ has an uniform
representation $\mathcal{F}$.
\end{enumerate}
%
%
\begin{remark} \label{}
Clearly every uniform representation is an antichain representation,
every antichain representation is a family representation, and every
family representation is a multifamily representation. Hence
$\omega_m(M) \leq \omega_f(M) \leq \omega_a(M) \leq \omega_u(M)$.
Note that given a multifamily representation $\{S_v\mid v\in V(M)\}$
of $M$ and a vertex subset $S\subseteq V(M)$, then $\{S_v\mid v\in
S\}$ form a multifamily representation of $\langle S \rangle_V$.
Thus we know that $\omega_m(M) \geq \omega_m(\langle S \rangle_V)$
for any $S\subseteq V(M)$. Similarly, $\omega_f(M) \geq
\omega_f(\langle S \rangle_V)$, $\omega_a(M) \geq \omega_a(\langle S
\rangle_V)$, and $\omega_u(M) \geq \omega_u(\langle S \rangle_V)$
for any $S\subseteq V(M)$.
\end{remark}

We may define the notion of uniqueness of representations of
intersection multigraphs. A multigraph $M$ is said to be {\bf
uniquely intersectable with respect to multifamilies (u.i.m.)} if
given a set $X$ with $|X|=\omega_m(M)$ and for any two families
$\alpha$ and $\beta$ of subsets of $X$ such that $\alpha$ and
$\beta$ are both multifamily representations of $M$, then $\beta$
can be obtained from $\alpha$ by a permutation of elements of $X$.
Similarly $M$ is {\bf uniquely intersectable with respect to
families (u.i.f.)} {\bf uniquely intersectable with respect to
antichains (u.i.a.)}, and {\bf uniquely intersectable with respect
to uniform families (u.i.u.)} are also defined.

Given a graph $G$, $Q\subseteq V(G)$ is said to be a \textit{clique}
of $G$ if every pair of distinct vertices $u,v$ in $Q$ are adjacent.
A \textit{clique partition} $\mathcal{Q}$ of a graph is a set of
cliques such that every pair of distinct vertices $u,v$ in $V(G)$
appears in exactly one clique in $\mathcal{Q}$ and for each
\textit{isolated vertex}, we need to use at least one
\textit{trivial clique} with only one vertex in $\mathcal{Q}$ to
cover it. The minimum cardinality of a clique partition of $G$ is
called the \textit{clique partition number} of $G$, and is denoted
by $cp(G)$. We refer to a clique partition of $G$ with the
cardinality $cp(G)$ as a \textit{minimum clique partition} of $G$.
Note that a clique partition $\mathcal{Q}$ of a graph $G$ gives rise
to a clique partition of $G-v$ by deleting the vertex $v$ from each
clique in $\mathcal{Q}$. Thus $cp(G)$ is not less than the clique
partition number of any induced subgraph of $G$.\\

In the following, we describe the correspondence between the
multifamily set representations and clique partitions of the edge
set for a multigraph $M=(V(M),E(M),q)$.

We first construct a clique partition
\begin{align*}
\mathcal{Q}=\{Q_1,...,Q_p\}
\end{align*}
of $M$, then with each clique $Q_k$ we associate an element $e_k$
and with each vertex $v_{\alpha}$ we associate a set
$S_{\mathcal{Q}}(v_\alpha)$ of elements $e_k$, where
\begin{align*}
e_k\in S_{\mathcal{Q}}(v_\alpha) \Leftrightarrow v_\alpha \in Q_k,
\end{align*}
that is, $S_{\mathcal{Q}}(v_\alpha)$ is the collection of those
elements for which the corresponding cliques contains $v_\alpha$.
Thus we obtain
\begin{align*}
\mathcal{F}(\mathcal{Q})\equiv\{S_{\mathcal{Q}}(v):v\in V(M)\}.
\end{align*}
Then clearly
\begin{align*}
\textbf{S}(\mathcal{F}(\mathcal{Q}))\equiv \bigcup_{v\in
V(M)}S_\mathcal{Q}(v)
\end{align*}
contains $p$ elements, and
\begin{align*}
|S_\mathcal{Q}(v_\alpha)\cap S_{\mathcal{Q}}(v_{\beta})|=q(v_{\alpha},v_{\beta}),
\end{align*}
since there is exactly $q(v_\alpha,v_\beta)$ cliques simultaneously
containing the two vertices $v_\alpha$ and $v_\beta$. Thus we have
constructed a multifamily representation
$$\mathcal{F}(\mathcal{Q})=\{S_{\mathcal{Q}}(v):v\in V(M)\}$$ from the clique partition $\mathcal{Q}$ of $M$, where
$$|\textbf{S}(\mathcal{F}(\mathcal{Q}))|\equiv |\bigcup_{v\in V(M)}S_\mathcal{Q}(v)|=p=|\mathcal{Q}|.$$

Conversely, given a multifamily representation
$\mathcal{F}=\{S_1,...S_n\}$ of $G$ with vertex set
$V(G)=\{v_1,...,v_n\}$, where $S_\alpha$ correspond to the set
attaching to $v_\alpha$, then we can also construct a clique
partition of $G$ by the following way.

Let $$\textbf{S}(\mathcal{F})\equiv
\bigcup_{\alpha=1}^nS_\alpha=\{e_1,...,e_p\}.$$ For each fixed $e_k$
in $\textbf{S}(\mathcal{F})$ we form a clique $Q_\mathcal{F}(e_k)$
using those vertices $v_\alpha$ such that the set $S_\alpha$
attaching to it contains $e_k$. Clearly each $Q_\mathcal{F}(e_k)$ is
indeed a clique of $G$. Thus we obtain
$$\mathcal{Q}(\mathcal{F})=\{Q_\mathcal{F}(e_1),...,Q_\mathcal{F}(e_p)\}.$$ and
\begin{align*}
q(v_\alpha,v_\beta)&=|S_\alpha \cap S_\beta|\\
= \mbox{the number of cliques in } \mathcal{Q}(\mathcal{F}) &\mbox{
simultaneously containing } v_\alpha \mbox{ and } v_\beta,
\end{align*}
since each element in $S_\alpha$ exactly represent a clique in
$\mathcal{Q}(\mathcal{F})$ containing $v_\alpha$. Thus we have
constructed a clique partition $\mathcal{Q}(\mathcal{F})$ of $M$
from the multifamily representation $\mathcal{F}$ of $M$, where
$$|\mathcal{Q}(\mathcal{F})|=p=|\bigcup_{\alpha=1}^nS_\alpha|\equiv |\textbf{S}(\mathcal{F})|.$$

From above, we may treat a graph $G$ as a special case of a
multigraph and hence we have the
correspondence for $G$, and in particular $\omega_m(G)=cp(G)$.\\

Intersection graphs and set representations of graphs were first
introduced and studied by E. Szpilrajn-Marczewski \cite{ESM} in
1945, and P. Erd\"{o}s, A. Goodman, and L. P\'{o}sa \cite{erdos} in
1966. The set representations in various senses such as family
representations, antichain representations, uniform representations,
and the associated uniqueness properties were studied in literatures
over the decades \cite{alter, bylka,tsuchiya2,tsuchiya, wang}. In
1997, Bylka and Komar \cite{bylka} tried to characterize the line
graphs with a unique multifamily representation, in other words, to
characterize all line graphs with a unique edge clique partition.
They studied the problem and solved for the characterization with
one case unsettled. In fact by S. McGuinness and R. Rees's results
in \cite{mcguinness2} (Theorem~\ref{Mc line thm} in this paper), one
may solve the remaining case and complete the whole characterization
of the uniqueness of the multifamily representation for line graphs,
which was already pointed out in T.-M. Wang's Ph.D. thesis
\cite{wang} in 1997.

In this paper, we study and completely classify the family and
antichain representations for line graphs, from which the associated
uniqueness results are derived.
In section 2 and section 3, we describe the relationship between the
theory of finite projective plane and edge clique partition of
complete graphs, then calculate the intersection numbers and
classify minimum set representations of a complete graph $K_n$ in
various senses. In following section 4 and section 5, we completely
determine the number of distinct minimum family representations and
minimum antichain representations of a line graph, respectively. In
section 6 we conclude with certain future research directions.

\section{\bf{Edge Clique Partitions and Finite Projective Spaces}}

A \textit{finite} \textit{linear sapce} $\Gamma
=(\mathcal{P},\mathcal{L})$ is a system consisting of a finite set
$\mathcal{P}$ of $n$ \textit{points} and a set $\mathcal{L}$ of
\textit{lines} satisfying the following axioms.

\begin{description}
\item [\quad (L1)] Any line has at least two points.
\item [\quad (L2)] Two points are on precisely one line.
\item [\quad (L3)] Any line has at most $n-1$ points.
\end{description}
If a space satisfy \textbf{(L1)} and \textbf{(L2)} but not \textbf{(L3)},
then clearly this space contain a unique line. This type of spaces is
referred to as \textit{trivial linear space}.

Suppose that $n\geq 3$. Let $\mathcal{Q}$ be a clique partition of
the complete graph $K_n$ such that each member of $\mathcal{Q}$ has
at least 2 and no more than $n-1$ vertices. Let
$\Gamma(\mathcal{Q})$ be the system whose set of points is the
vertex set of $K_n$, and whose lines are the members of
$\mathcal{Q}$. Incidence is defined as following. A points $v$ is
incident with a line $Q$ if $v$ is a vertex of $Q$. Then
$\Gamma(\mathcal{Q})$ is a finite linear space. Conversely, if
$\Gamma$ is a finite linear space on $n$ points, then there is a
clique partition $\mathcal{Q}$ of $K_n$ such that $\Gamma
=\Gamma(\mathcal{Q})$, where each member of $\mathcal{Q}$ has at
least 2 and no more than $n-1$ vertices.

Thus there is a one-one correspondence between all clique partitions
of the complete graph $K_n$ by cliques with cardinality at least 2
and at most $n-1$, and all finite linear spaces with $n$ points.

A \textit{projective plane} is a finite linear spaces $\Pi$
satisfying further the following two axioms.

\begin{description}
\item [\quad (P1)] Any two distinct lines have a point in common.
\item [\quad (P2)] There are four points, no three of which are on a common line.
\end{description}

Suppose that $\Pi$ is a projective plane with a finite number $n$ of
points and a finite number $l$ of lines. Then it is probative that
for some integer $k\geq 2$, $n=l=k^2+k+1$, and $\Pi$ has point and
line regularity $k+1$, where each point is on exactly $k+1$ lines
and each line contains exactly $k+1$ points. We call such a number
$k$ the order of the projective plane. Besides, any two lines in a
projective plane intersect on a common point, or paraphrased into
terms of clique partition, any two cliques intersect on a common
vertex.

The smallest projective plane has order $k=2$, which is the
\textit{Fano Plane}, as illustrated in Figure 1. Note that the
segments (straight or round) passing through
$\{a,b,c\},\{c,d,e\},\{a,f,e\},\{a,g,d\},\{b,g,e\},\{f,g,c\},\{b,d,f\}$
respectively stand for seven lines.

\begin{figure}[h]
\label{fano} \centering
\includegraphics[width=0.45\textwidth]{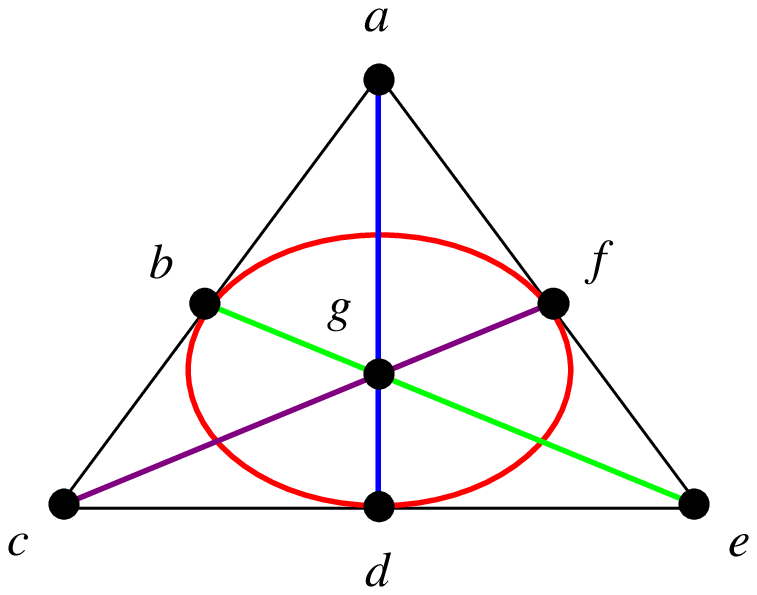}
\caption{Fano Plane}
\end{figure}

It is well known that there are unique projective planes of orders
2, 3, 4, 5, 7, and 8 respectively while there are none of order 6
and 10 \cite{lam}. There are at least 4 non-isomorphic projective
planes of order 9, but no one yet know the exactly number. It is
also known that one may construct a finite projective plane of order
$p^k$, where $p$ is prime, based on the theory of finite fields.

In 1948, de Bruijn and Erd\"{o}s proved a theorem about linear space
which we paraphrase in terms of clique partition as follows.

\begin{thm} \label{deBr} {\bf (de Bruijn and Erd\"{o}s\cite{bruijn}, 1948)}
If $\mathcal{Q}$ with $|\mathcal{Q}|>1$ is a clique partition of
$K_n$ with $n\geq 3$, no one of whose members is a trivial clique,
that is, the clique consisting of one single vertex, then
$|\mathcal{Q}|\geq n$, with equality if and only if

\begin{description}
\item [\quad (a)] $\mathcal{Q}$ consists of one clique on $n-1$ vertices and $n-1$ copies of $K_2$,
\item [\quad \mbox{}] or
\item [\quad (b)] The finite linear space corresponding to $\mathcal{Q}$ is a projective plane.
\end{description}
\end{thm}

Note that the linear spaces corresponding to the class of clique
partitions in (a) are traditionally referred to as
\textit{near-pencil} in finite linear space theory. The example for
the above Theorem~\ref{deBr} can refer to the one that the
nontrivial clique partitions of $K_7$ use at least 7 cliques, and
the extremal cases happen when they consist of either six ${K_2}'s$
with one $K_6$, or seven ${K_3}'s$ corresponding to the Fano plane.

\section{\bf{Intersection Representations of Complete Graphs}}

In this section, we calculate the intersection numbers and study the
uniqueness of the minimum representations for complete graphs $K_n$
in various senses. We start from the family representations.

\begin{thm} \label{intersect num K_n}
Let $K_n$ be a complete graph on $n$ vertices, $n\geq 3$.
\begin{enumerate}
\item $\omega_{f}(K_n) = n$.
\item $K_n$ is not \textit{u.i.f.}
\end{enumerate}
\end{thm}

\begin{proof}
For any complete graph $K_n$ with $n\geq 3$, we can always construct
a family representation by the following method. Take an element,
say $e_1$ common to the representation sets of all vertices. Then
attach elements $e_2,...,e_{n-1}$ to some $n-1$ vertices of the $n$
vertices, respectively. On the other hand, there cannot exist a
representation $\mathcal{F}$ of $K_n$ with
$|\textbf{S}(\mathcal{F})|\leq n-1$, for otherwise we can first
delete all elements in $\textbf{S}(\mathcal{F})$ that appear in the
representation set of only one vertex, which we would referred to as
\textit{monopolized element} in the rest of this paper, from the
representation sets of all vertices and say the resulting
representation $\mathcal{F}'$. Note that $\mathcal{F}'$ is a
multifamily representation of $K_n$, since monopolized elements have
nothing to do with multifamily representation of a multigraph. Now
we take $\mathcal{Q}(\mathcal{F}')$. Note that
$|\mathcal{Q}(\mathcal{F}')|\leq n-1$. Clearly
$\mathcal{Q}(\mathcal{F}')$ is a clique partition of $K_n$
containing no trivial clique. By theorem~\ref{deBr} and the fact
that $|\mathcal{Q}(\mathcal{F}')|\leq n-1$, we know that
$|\mathcal{Q}(\mathcal{F}')|=1$, that is,
$\mathcal{Q}(\mathcal{F}')$ consists of only one clique, containing
all $n$ vertices of $K_n$. But clearly we cannot recover
$\mathcal{F}$ from
$\mathcal{F}(\mathcal{Q}(\mathcal{F}'))=\mathcal{F}'$ by adding
monopolized elements to the members of $\mathcal{F}'$, since
$\mathcal{F}$ is a representation of $K_n$ with
$|\textbf{S}(\mathcal{F})|\leq n-1$, a contradiction. From above we
know that $\omega(K_n)=n$.

Now we investigate the uniqueness of $K_n$'s representation. Assume
a representation $\mathcal{F}$ of $K_n$ with
$|\textbf{S}(\mathcal{F})|=n$. Delete all monopolized elements in
$\textbf{S}(\mathcal{F})$ from the representation sets of all
vertices, say the resulting representation $\mathcal{F}'$, and then
take $\mathcal{Q}(\mathcal{F}')$. Now
$|\mathcal{Q}(\mathcal{F}')|\leq n$. Clearly
$\mathcal{Q}(\mathcal{F}')$ is a clique partition of $K_n$
containing no trivial clique. By theorem~\ref{deBr} and
$|\mathcal{Q}(\mathcal{F}')|\leq n$, we know that
$\mathcal{Q}(\mathcal{F}')$ consists of only one clique, or is a
near-pencil or projective plane. If $\mathcal{Q}(\mathcal{F}')$ is a
near-pencil or projective plane, then
$|\mathcal{Q}(\mathcal{F}')|=n$ and thus it is clear that in these
two cases we had never deleted any monopolized element from the
representation set of any vertex when we proceed from $\mathcal{F}$
to $\mathcal{F}'$. Thus in these two cases, the original
representation $\mathcal{F}$ is just
$\mathcal{F}(\mbox{near-pencil})$ or $\mathcal{F}(\mbox{projective
plane})$, note that here and in the following we use these two
terminologies "near-pencil" and "projective plane" to stand for
their corresponding clique partitions, respectively. Clearly these
two representations indeed have their constituting sets pairwise
distinct.

For the remaining case, $\mathcal{Q}(\mathcal{F}')$ consists of only
one clique. Thus in this case we must had deleted $n-1$ monopolized
elements in proceeding from $\mathcal{F}$ to $\mathcal{F}'$. And
clearly all constituting sets of $\mathcal{F}$ has a common element,
say $e_1$, and some $n-1$ constituting sets of $\mathcal{F}$ have
monopolized elements, say $e_2,...,e_{n-1}$, respectively.

Thus in above we have proved that every complete graph $K_n$ with
$n\geq 3$ has intersection number $n$ and has three manners for
forming its minimum representations. Note that the practicability of
the one manner derived from projective plane depends on whether or
not $n=k^2+k+1$ for some $k\geq 2$ and there exists projective plane
of order $k$.
\end{proof}

Then we investigate the minimum antichain representations of
complete graphs.

\begin{thm}\label{anti intersect num K_n}
Let $K_n$ be a complete graph on $n$ vertices, where $n\geq 3$.
\begin{enumerate}
\item $\omega_{a}(K_n) = n$.
\item $K_n$ is not \textit{u.i.a.}
\end{enumerate}
\end{thm}

\begin{proof}
Because $\mathcal{F}(\mbox{near-pencil})$ itself is a antichain
representation of $K_n$ making use of $n$ elements, we know that
$\omega_{a}(K_n)\leq n$. Assuming an antichain representation
$\mathcal{F}$ of $K_n$ with $|\textbf{S}(\mathcal{F})|\leq n$.
Delete all monopolized elements in $\textbf{S}(\mathcal{F})$ from
the representation set of all vertices, say the resulting
representation $\mathcal{F}'$, and then take
$\mathcal{Q}(\mathcal{F}')$. Now $|\mathcal{Q}(\mathcal{F}')|\leq n$
and $\mathcal{Q}(\mathcal{F}')$ is a clique partition of $K_n$ with
no trivial clique. By theorem~\ref{deBr}, we know that
$\mathcal{Q}(\mathcal{F}')$ have only one member, or is a
near-pencil, or projective plane. Clearly
$\mathcal{F}(\mbox{near-pencil})$ and $\mathcal{F}(\mbox{projective
plane})$ are both antichain representation. As for the remaining
case that $\mathcal{Q}(\mathcal{F}')$ have only one member, we
cannot recover $\mathcal{F}$ from
$\mathcal{F}(\mathcal{Q}(\mathcal{F}'))=\mathcal{F}'$ by adding
monopolized elements to the members of $\mathcal{F}'$, since
$\mathcal{F}$ is an antichain representation of $K_n$ with
$|\textbf{S}(\mathcal{F})|\leq n$.

Thus we have proved that every complete graph $K_n$ with $n\geq 3$
has antichain intersection number $n$ and has two manners for
forming its minimum antichain representations with the one manner
derived from projective plane being provisory upon the existence of
projective plane of appropriate order.
\end{proof}


As for the investigation of the minimum uniform representations of
complete graphs, we shall refer to the following theorem due to W.
G. Bridges.

\begin{thm}\label{bridges}{\bf (W. G. Bridges\cite{bridges}, 1972)}
Let $\Gamma=(\mathcal{P},\mathcal{L})$ be a finite linear space with
$n\neq 5$ points and $l$ lines. Then $l=n+1$ if and only if $\Gamma$
is a projective plane with one point removed from $\mathcal{P}$ and
every line of $\mathcal{L}$. As for the case of $n=5$, please see
Figure 2.
\end{thm}

\begin{figure}
\centering
\includegraphics[width=0.6\textwidth]{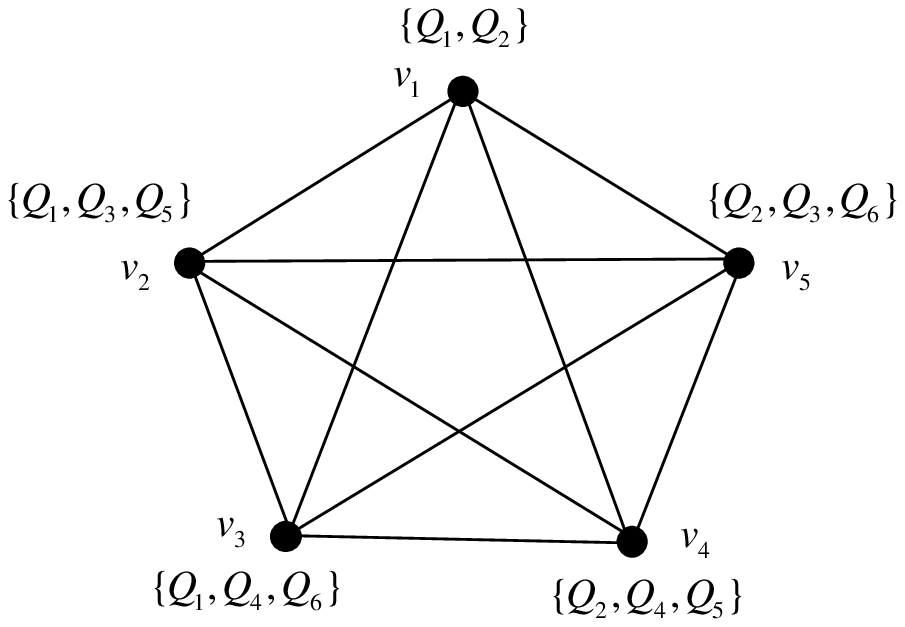}
\caption{Finite linear space with $5$ points and $6$ lines}
\end{figure}

\begin{thm}\label{uni intersect num K_n} We have the following
cases for the uniform representations of complete graphs:
\begin{enumerate}
\item for $n=3$, $K_n$ has uniform intersection number $n$ and
has only one manner to form its minimum uniform representations;
\item for $n\geq 4$, where $n=k^2+k+1$ for some $k\geq 2$ so that
there exists projective plane of order $k$, $K_n$ has uniform
intersection number $n$ and has only one manner to form its minimum
uniform representations;
\item for $n\geq 4$ where $n=k^2+k$ for
some $k\geq 2$ so that there exists projective plane of order $k$,
$K_n$ has uniform intersection number $n+1$ and has two manners to
form its minimum uniform representations; and
\item for $n\geq 4$
where $n\neq k^2+k+1$ and $n\neq k^2+k$ for all $k\geq 2$, $K_n$ has
uniform intersection number $n+1$ and has only one manner to form
its minimum uniform representations.
\end{enumerate}
\end{thm}

\begin{proof}
Assume an uniform representation $\mathcal{F}$ of $K_n$ with
$|\textbf{S}(\mathcal{F})|\leq n$. Delete all monopolized elements
in $\textbf{S}(\mathcal{F})$ from the representation sets of all
vertices, say the resulting representation $\mathcal{F}'$, and then
take $\mathcal{Q}(\mathcal{F}')$. Now
$|\mathcal{Q}(\mathcal{F}')|\leq n$ and $\mathcal{Q}(\mathcal{F}')$
is a clique partition of $K_n$ with no trivial clique. By
Theorem~\ref{deBr}, we know that $\mathcal{Q}(\mathcal{F}')$ have
only one member, or is a near-pencil, or a projective plane. Clearly
$\mathcal{F}(\mbox{projective plane})$ is an uniform representation
in its own right, while we cannot recover an uniform representation
$\mathcal{F}$, with $|\textbf{S}(\mathcal{F})|\leq n$, of $K_n$ with
$n\geq 3$ from $\mathcal{F}(\mbox{near-pencil})$ by adding
monopolized elements to the members of it except possibly $n=3$. And
for the remaining case, that is, $\mathcal{Q}(\mathcal{F}')$ has
only one clique, we also cannot recover $\mathcal{F}$ from
$\mathcal{F}(\mathcal{Q}(\mathcal{F}'))$.

Thus whenever $n\geq 4$, we have that $\omega_{u}(K_n)=n$ and $K_n$
is u.i.u. if and only if $n=k^2+k+1$ for some $k\geq 2$ and there
exists projective plane of order $k$.

In case that $4\leq n\neq k^2+k+1$ or there exists no projective
plane of order $k$, since we can always form an uniform
representation $\mathcal{F}$ of $K_n$ with
$|\textbf{S}(\mathcal{F})|=n+1$ by first adopting an element common
to the representation set of all vertices and then for the
representation set of each vertex attaching a monopolized element to
it. Thus for this case we have $\omega_{u}(K_n)=n+1$. Now given an
uniform representation $\mathcal{F}$ of $K_n$ with
$|\textbf{S}(\mathcal{F})|\leq n+1$, first we delete all monopolized
elements of $\textbf{S}(\mathcal{F})$ from the representation set of
each vertex resulting in another representation, say $\mathcal{F}'$,
and then take $\mathcal{Q}(\mathcal{F}')$. Now
$|\mathcal{Q}(\mathcal{F}')|\leq n+1$. By theorem~\ref{deBr} and
\ref{bridges}, (note that we have assumed that $n\geq 4$ and there
exists no projective plane of appropriate order) and the fact that
we cannot recover $\mathcal{F}$ from $\mathcal{F}(\mbox{near-pencil
with } n\geq 4 \mbox{ vertices})$ or $\mathcal{F}(\mbox{the clique
partition as in Figure 5})$, we know that
$\mathcal{Q}(\mathcal{F}')$ either consists of only one clique, or
is a projective plane with one vertex deleted. The corresponding
representation of the latter is an uniform representation in its own
right and we can easily recover $\mathcal{F}$ from the corresponding
representation set of the former by returning monopolized element to
each member of $\mathcal{F}$.
\end{proof}

\section{\bf{Family Representations of Line Graphs}}

The \textit{line graph} of a graph $G$, which we assume to be
finite, undirected and simple in this paper, written as $L(G)$, is
the graph whose vertices are the edges of $G$, with its two vertices
adjacent if and only if the two edges in $G$ corresponding to these
two vertices have a common endpoint in $G$.

For each vertex $v$ of $G$, the set $e_v$ consisting of all edges in
$G$ containing $v$ induces a maximal clique in $L(G)$. This is one
of the only two types of maximal cliques in $L(G)$, while the rest
of maximal cliques is induced by triangles in $G$. Besides, any edge
$ef\in E(L(G))$ with $e=uv$ and $f=vw$ being two edges in $G$ can
only be contained in either a clique induced by $e,f$ possibly
together with some edges in $G$ with $v$ as endpoint or the clique
induced by the triangle $uvw$ in $G$ (if $u$ is adjacent to $w$).
Clearly the set $P=\{e_v:v\in G, d(v)\geq 2\}$ is a clique partition
of $L(G)$ which we will call the \textit{canonical clique partition}
of $L(G)$. Note that each vertex of $L(G)$ is contained in exactly
two cliques in $P$.

Let $G$ be a graph. A \textit{wing} in $G$ is a triangle with the
property that exactly two of its vertices have degree two in $G$,
while a \textit{3-wing} is a wing with the vertex in it having
degree greater than two having degree exactly three. Besides, we
define a \textit{star} in $G$ to be a collection of edges in $G$
which intersect on a common vertex. Note that a star need not
consist of all edges incident with some vertex, but only a
sub-collection of those edges. We will use the notation $S_v^i$ to
indicate a star with $i$ edges, centered at $v$. The \textit{join}
of simple graphs $G$ and $H$, denoted $G\vee H$, is the graph
obtained from the vertex-disjoint union $G+H$ by adding all the
edges $\{xy:x\in V(G),y\in V(H)\}$. We denote the graph by $W_t$,
$t\geq 2$, as in Figure 3.

\begin{figure} \centering
\includegraphics[width=0.6\textwidth]{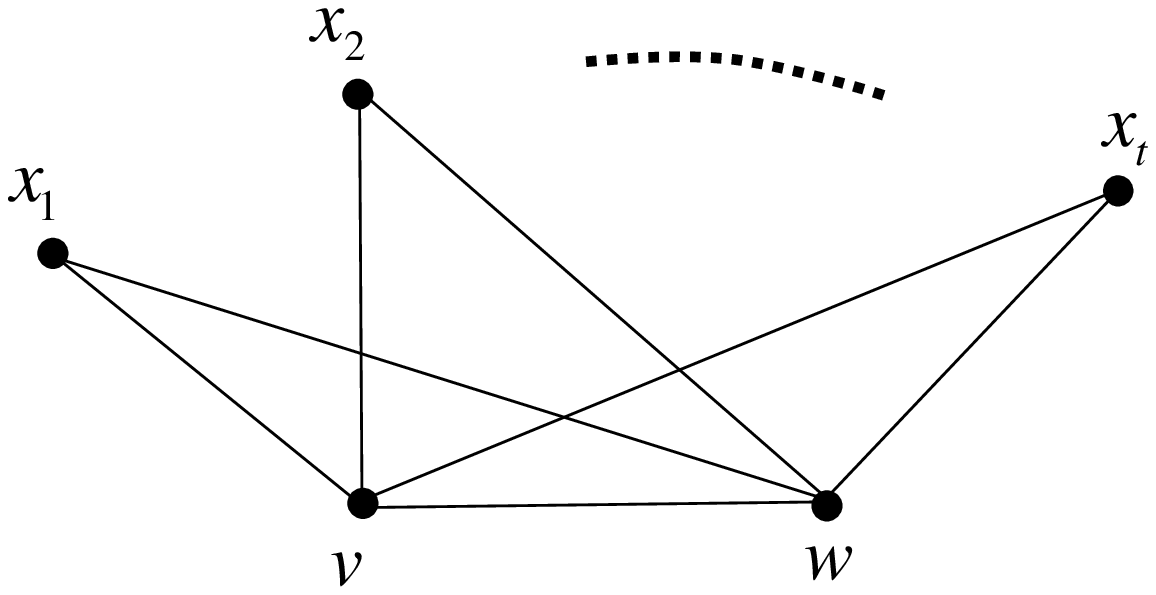} \caption{The graph $W_t$ with $t \geq 2$}
\end{figure}

S. McGuinness and R. Rees proved the following theorem to count the
number of distinct minimum edge clique partitions.

\begin{thm}\label{Mc line thm}{\bf (S. McGuinness and R. Rees\cite{mcguinness2}, 1990)}
Let $G$ be a connected graph, and $G \neq K_3,K_4,(K_2+K_2+K_2)\vee
K_1$ (or $3K_2\vee K_1$ in abbreviation), $W_t$ with $t\geq 2$. Let
$V_2(G)$ denote the set of vertices in $G$ with degree at least two
, and let $w_3$ denote the number of 3-wing in $G$. Then
$cp(L(G))=|V_2(G)|$ and there are exactly $2^{w_3}$ distinct minimum
clique partitions of $L(G)$.
\end{thm}

A cursory illustration of the above theorem here would be
advantageous for our further study. Note that the above theorem
wouldn't concern itself with "isomorphism", that is, it would regard
two clique partitions to be distinct if the cliques in two clique
partitions are not derived from the same stars and triangles in $G$.
For illustration, to look up a minimum clique partition of the line
graph of the graph $G$ in Figure 4, we have two "distinct manners",
one by the upper triangle and the inferior two edges in $L(G)$, that
is, by the three stars in $G$ centered at $u,v,w$, whereas the other
by the inferior triangle and the upper two edges in $L(G)$, that is,
by the 3-wing $uvw$ and the two stars $\{vw, vy\}$, $\{vu, vy\}$ in
$G$.

\begin{figure} \centering
\includegraphics[width=0.5\textwidth]{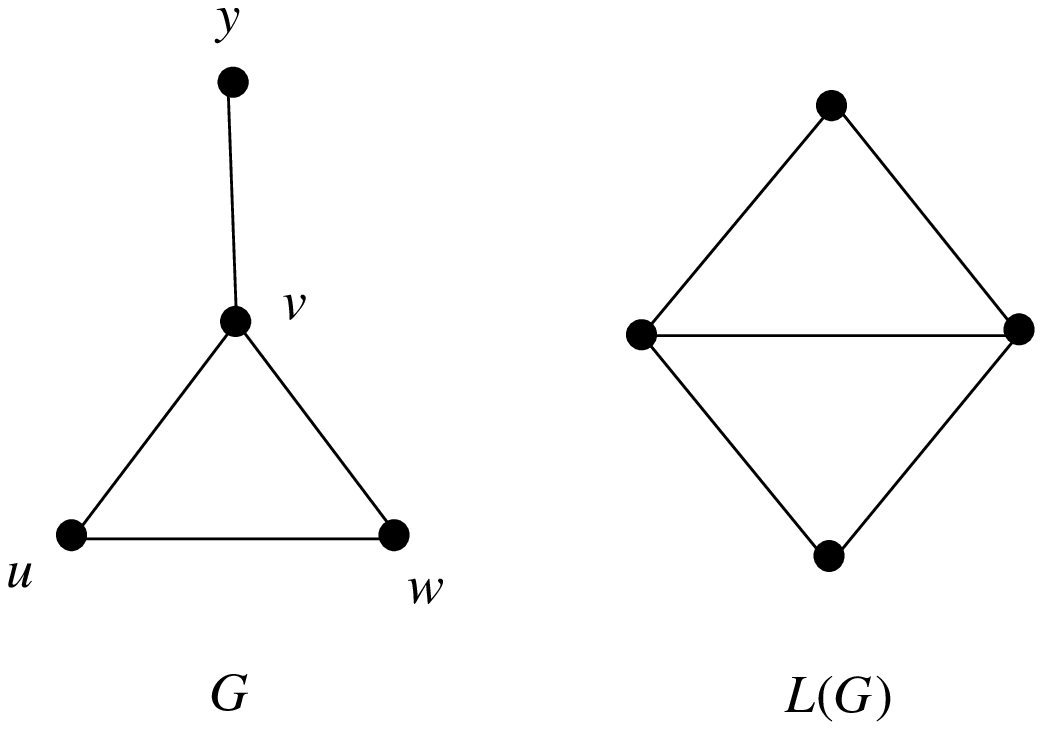} \caption{A graph and its line graph}
\end{figure}

We will follow this criterion when deciding whether or not two
clique partitions are the same. The above theorem clarify the fact
that to attain a minimum clique partition of $L(G)$, where note that
$G$ is the class of graphs aforementioned in the above theorem, no
triangle in $L(G)$ induced by a triangle in $G$ other than 3-wing
can be used. And the adopting in a clique partition of $L(G)$ of any
triangle induced by one 3-wing in $G$ can also yield a minimum
clique partition other than the unique other minimum clique
partition, called the canonical one, which consists of all maximal
cliques of $L(G)$ induced by one maximal star in $G$. Thus each
3-wing in $G$, refer to Figure 4, yield two distinct clique
partitions of $L(G)$, one adopting the upper triangle and the
inferior two edges in the right graph of Figure 4, while the other
adopting the inferior triangle and the upper two edges; and
therefore as the aforementioned by the above theorem $G$ has exactly
$2^{w_3}$ distinct minimum clique partitions.

Now we investigate the intersection number of the line graph $L(G)$
of a connected simple graph $G \neq K_3,K_4,3K_2\vee K_1, \mbox{or }
W_t, t\geq 2$.

\begin{thm}\label{linegraph representation}
Let $G$ be a connected simple graph, and $G \neq K_3,K_4,3K_2\vee
K_1, \mbox{or } W_t, t\geq 2$. In addition, we suppose that $G$ is
not a star. Let $V_2(G)$ denote the set of vertices in $G$ with
degree at least two, and let $w_3$ denote the number of 3-wing in
$G$. And let $u_1^{(i)},...,u_{m_i}^{(i)}$ be all vertices in $G$ of
degree one and adjacent to $v_i$ with $d(v_i)>1$. We suppose that
there are $k$ vertices with its degree more than one in $G$ in total
which are adjacent to some vertex of degree one, i.e., $1\leq i\leq
k$. Then $\omega_{f}(L(G))=|V_2(G)|+\sum_{i=1}^k(m_i-1)$ and there
are exactly $2^{w_3}$ distinct minimum representations of $L(G)$.
\end{thm}

\begin{proof}
First we consider the following question: When do a minimum clique
partition, say $\mathcal{Q}$, of $L(G)$ has two vertices obtaining
the same representation set after we take
$\mathcal{F}(\mathcal{Q})$? Clearly if such two vertices, say
$e_1,e_2$, exist, then their two corresponding edges in $G$, say
$vu_1,vu_2$, inetrsect and either $d(u_1)=d(u_2)=1$ or $vu_1u_2$ is
a wing in $G$ with $d(u_1)=d(u_2)=2$.

\begin{figure} \centering
\includegraphics[width=0.5\textwidth]{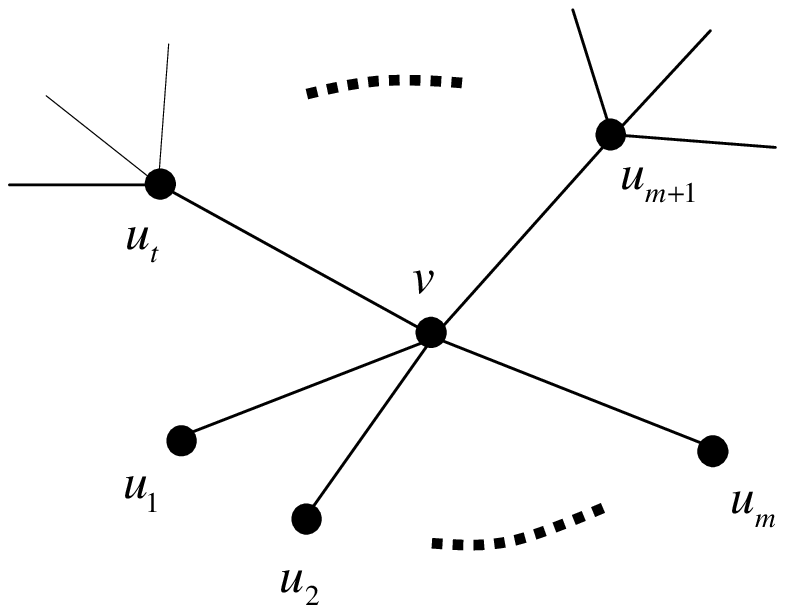} \caption{Example}
\end{figure}

For the former case see Figure 5, where for the sake of generality
we suppose that $u_1,...,u_m$ are vertices in $G$ with degree one
and $u_{m+1},...,u_t$ with degree at least 2. Immediately after we
ask the question whether or not we can represent the complete
subgraph $K_m$ in $L(G)$, refer to Figure 5, with vertex set
$\{vu_1,...,vu_m\}$ by exactly $m$ elements in some minimum
representation of $L(G)$. (Note that it is impossible to represent
it by $m-1$ elements.) Assuming that we can, then this $K_m$'s
representation can correspond to three types of clique partitions,
say the corresponding clique partition being $\mathcal{Q}$, that is,
near-pencil, projective plane, or $K_m$ together with $m-1$ trivial
cliques. Note that projective plane and near-pencil have a common
property, that is, any two lines intersect on a common point , or
paraphrased into terms of clique partition, any two cliques
intersect on a common vertex, and recall that in the method by which
we construct a correspondence between multifamily representation and
clique partition, an element in multifamily representation
correspond to a clique in clique partition. Thus for the former two
cases, to make $vu_{m+1},...,vu_t$ be adjacent to $vu_1,...,vu_m$,
we shouldn't rely on more than one element in
$\bigcup_{i=1}^mS_{\mathcal{Q}}(vu_i)$, since for any two elements,
the vertex on which the two cliques respectively corresponding to
them intersect has its representation set comprising them. Nor
should we use one. (Unless $G$ itself is a star, that is, $t=m$.)
But for the third case we can use the element in
$\bigcup_{i=1}^mS_{\mathcal{Q}}(vu_i)$ corresponding to the clique
$K_m$ (and note that this is the unique approach if we would like
not to use new elements).

\begin{figure} \centering
\includegraphics[width=0.3\textwidth]{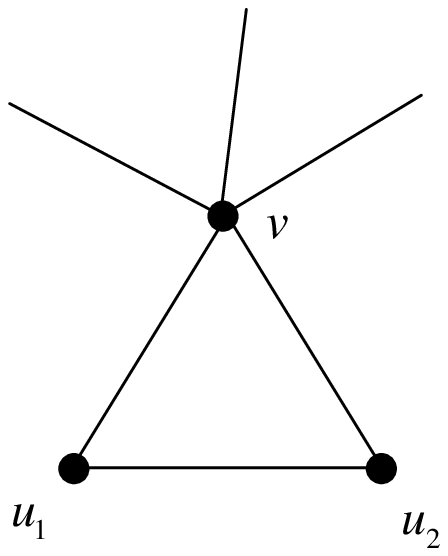} \caption{Example}
\end{figure}

On the other hand, for the case that $vu_1u_2$ is a wing in $G$ with
$d(u_1)=d(u_2)=2$, refer to Figure 6. In this case, whether or not
$vu_1u_2$ is a 3-wing in $G$, that is, whether or not the adopting
of the triangle in $L(G)$ induced by the triangle $vu_1u_2$ in $G$
can occur in one minimum clique partition of $L(G)$, we can't have
$S_{\mathcal{Q}}(vu_1)=S_{\mathcal{Q}}(vu_2)$.
\end{proof}

The case that $G=K_3$ or a star is not hard to see. For $G=3K_2\vee
K_1$, S. McGuinness and R. Rees \cite{mcguinness2} have shown that
$L(G)$ admits exactly three distinct minimum clique partitions, and
with a little direct inspection we see that these three partitions
correspond to three distinct minimum (antichain) representations
respectively. (As a matter of fact, two of the three are
isomorphic.)

As for $G=K_4$, it is easily verified that there are exactly two
distinct but in fact isomorphic clique partitions, one by all the
cliques in $L(G)$ induced by some maximal star in $G$, while the
other by all the triangles in $L(G)$ induced by some triangle in
$G$, and with a little direct inspection we see that these two
partitions correspond to two distinct minimum (antichain)
representations respectively.

As for $G=W_t,t\geq 2$, S. McGuinness and R. Rees \cite{mcguinness2}
have shown that $L(G)$ has exactly two distinct minimum clique
partitions, and with a little direct inspection we see that these
two partitions correspond to two distinct minimum (antichain)
representations respectively.

\section{\bf{Antichain Representations of Line Graphs}}

Next we consider the antichain intersection number of the line graph
$L(G)$, where $G$ is connected simple and $\neq K_3,K_4,3K_2\vee
K_1, \mbox{or } W_t, t\geq 2$.

\begin{thm}\label{linegraph anti representation}
Let $G$ be a connected simple graph, and $G \neq K_3,K_4,3K_2\vee
K_1, \mbox{or } W_t, t\geq 2$. In addition, we suppose that $G$ is
not a star, and is not a graph as in Figure 11, 12, 13, 14. Let
$V_2(G)$ denote the set of vertices in $G$ with degree at least two,
and let $w_3$ denote the number of 3-wing in $G$. And let
$u_1^{(i)},...,u_{m_i}^{(i)}$ be all vertices in $G$ of degree one
and adjacent to $v_i$ with $d(v_i)>1$. We suppose that there are $k$
vertices with its degree more than one in $G$ in total which are
adjacent to some vertex of degree one, i.e., $1\leq i\leq k$. And we
suppose that there are altogether $k'$ such numbers $i$ in
$\{1,...,k\}$ so that $t_i=m_i+1$, and that among the $k'$ numbers
there are $k''$ such numbers $i$ so that there exists projective
plane with $t_i$ vertices. Then when regarding $L(G)$ as a
multigraph, $\omega_{a}(L(G))=|V_2(G)|+\sum_{i=1}^km_i$ and there
are exactly $3^{k'-k''}4^{k''}$ distinct minimum antichain
representations of $L(G)$.
\end{thm}

\begin{proof}
First, we consider the question that when do a minimum clique
partition, say $\mathcal{Q}$, of $L(G)$ has two vertices the two
corresponding representation sets for which after we take
$\mathcal{F}(\mathcal{Q})$ would have one contained in the other.
Clearly, the two edges in $G$, say $e_1,e_2$, corresponding to such
two vertices must intersect, say $e_1=vu_1,e_2=vu_2$, and one of
$u_1,u_2$, say $u_1$ throughout the rest of this paper, has no
neighbor other than $v,u_2$.

We first consider exclusively the case that $vu_1u_2$ form a
triangle in $G$. Now $d(u_1)=2$. If only we have never made use of
the clique in $L(G)$ induced by the triangle $vu_1u_2$ in $G$ in a
clique partition, say $\mathcal{Q}$, of $L(G)$, we utterly needn't
to worry about the inclusion relation between the two representation
sets $S_{\mathcal{Q}}(e_1),S_{\mathcal{Q}}(e_2)$. Thus what we need
to consider is mere the case that there exists a minimum clique
partition of $L(G)$ making use of the triangle in $L(G)$ induced by
the triangle $vu_1u_2$ in $G$, i.e., that the triangle $vu_1u_2$ is
a 3-wing. Recall that we have supposed that $d(u_1)=2$, and thus
exactly one of $v,u_2$ has degree two and the other has degree
three. In case that $d(v)=2$, making use of the triangle $vu_1u_2$
in a minimum clique partition, say $\mathcal{Q}$, will make
$S_{\mathcal{Q}}(e_1)$ be contained in $S_{\mathcal{Q}}(e_2)$. Thus
in this case the representation derived from the minimum clique
partition of $L(G)$ making no use of the triangle $vu_1u_2$, i.e.,
the canonical one, is the unique approach to form an minimum
antichain representation of $L(G)$. In case that $d(u_2)=2$, whether
or not we make use of the triangle $vu_1u_2$ in a minimum clique
partition, say $\mathcal{Q}$, of $L(G)$, there can't be inclusion
relation between $S_{\mathcal{Q}}(e_1),S_{\mathcal{Q}}(e_2)$. But if
we make use of the triangle $vu_1u_2$, then
$S_{\mathcal{Q}}(u_1u_2)$ will be contained in both
$S_{\mathcal{Q}}(e_1)$ and $S_{\mathcal{Q}}(e_2)$. Thus in this case
we have the same conclusions as the former one.

Now what remained is the case that $u_1$ is not adjacent to $u_2$.
For this case, we can without loss of generality assume that
$d(u_1)=1$ while leave $d(u_2)$ unappointed. See Figure 5, where for
the sake of generality we suppose that $u_1,...,u_m$ are vertices in
$G$ with degree one and $u_{m+1},...,u_t$ with degree at least two.
Immediately after we look for a minimum antichain representation of
$L(G)$ in which the complete subgraph $K_m$ with vertex set
$vu_1,...,vu_m$ is represented using exactly $m$ elements. (Note
that it is impossible to represent it by $m-1$ elements.) Assuming
that we can, then this $K_m$'s representation can only correspond to
two types of clique partitions, say the corresponding clique
partition being $\mathcal{Q}$, that is, near-pencil or projective
plane. (When $m=1$, we can represent $K_m$ by $m$ elements with
respect to antichain. But in this case we can't make $u_1$ be
adjacent to $u_{m+1},...,u_t$ by the single element in the
representation set of $u_1$ so that the representation set of $u_1$
wouldn't be contained in the representation sets of
$u_{m+1},...,u_t$, unless $t=1$, that is, $G=K_2$.) Now to make
$vu_{m+1},...,vu_t$ be adjacent to $vu_1,...,vu_m$, we can't use
more than one element in $\bigcup_{i=1}^mS_{\mathcal{Q}}(vu_i)$ for
securing the representation sets of any two vertices from
overlapping on more than one element, neither can we use one (unless
$G$ itself is a star, that is, $t=m$).

Thus we should yield by one step looking for a minimum antichain
representation of $L(G)$ in which the complete subgraph $K_m$ with
vertex set $\{vu_1,...,vu_m\}$ is represented by exactly $m+1$
elements. Assuming such a minimum antichain representation, then by
theorem~\ref{deBr} and \ref{bridges} this $K_m$'s representation can
only correspond to five types of clique partitions, say the
corresponding clique partition being $\mathcal{Q}$, that is,
near-pencil together with one trivial clique attached on it,
projective plane together with one trivial clique attached on it,
one $K_m$ together with $m$ trivial cliques attached on it, one as
in Figure 2, or projective plane with one vertex deleted.

For the first case, to make $vu_{m+1},...,vu_t$ be adjacent to
$vu_{1},...,vu_m$, we can't use more than one element in
$\bigcup_{i=1}^mS_{\mathcal{Q}}(vu_i)$ different from the
monopolized one for securing the representation sets of any two
vertices from overlapping on more than one element. Since we have
only one monopolized element in
$\bigcup_{i=1}^mS_{\mathcal{Q}}(vu_i)$, thus we must try to use one
non-monopolized element in $\bigcup_{i=1}^mS_{\mathcal{Q}}(vu_i)$ to
make $m-1$ vertices of the $K_m$ be adjacent to $vu_{m+1},...,vu_t$
(unless $G$ is a star, i.e., $t=m$), and then use the monopolized
element on the vertex of the $K_m$ other than the aforementioned
$m-1$ vertices to make this vertex be adjacent to
$vu_{m+1},...,vu_t$. Clearly we have only one approach to do so,
that is, first take the  element in
$\bigcup_{i=1}^mS_{\mathcal{Q}}(vu_i)$ that correspond to the clique
$K_{m-1}$ in $\mathcal{Q}$ to make all vertices on this  $K_{m-1}$
be adjacent to $vu_{m+1},...,vu_t$, and then use the monopolized
element on the vertex not on this $K_{m-1}$ to make this vertex be
adjacent to $vu_{m+1},...,vu_t$. But when $t>m+1$, using this method
will make $|S(u_{m+1})\cap S(u_{m+2})|\geq 2$. Thus, provided that
$G$ is not a star, this method can be carried out only if $t=m+1$.

As for the second case, i.e. projective plane together with one
trivial clique attached on it, similarly we must try to use one
non-monopolized element in $\bigcup_{i=1}^mS_{\mathcal{Q}}(vu_i)$ to
make $m-1$ vertices of the $K_m$ be adjacent to $vu_{m+1},...,vu_t$
(unless $G$ is a star, i.e., $t=m$). But we know that in a
projective plane of order $k$ each clique contain $k+1$ vertices,
whereas there are $k^2+k+1$ vertices in total where $k\geq 2$, and
thus each clique in a projective plane has
$$(k^2+k+1)-(k+1)=k^2\geq 4$$ vertices not on it. Thus in this case we have failed.

For the third case, i.e., one clique $K_m$ together with $m$ trivial
cliques attached on it, for the sake not to make two representation
sets overlap on more than one element, we can only use the element
in $\bigcup_{i=1}^mS_{\mathcal{Q}}(vu_i)$ corresponding to the
clique $K_m$ in $\mathcal{Q}$ or all monopolized elements in
$\bigcup_{i=1}^mS_{\mathcal{Q}}(vu_i)$ to make $vu_{m+1},...,vu_t$
be adjacent to $vu_1,...,vu_m$. But when $t>m+1$ and there is one
vertex, say $u_{m+1}$, in $\{u_{m+1},...,u_t\}$ which is not
adjacent to any other vertex in $\{u_{m+1},...,u_t\}$, then for the
sake that we should make $vu_{m+1}$ be adjacent to
$vu_{m+2},...,vu_t$, we can only use the element in
$\bigcup_{i=1}^mS_{\mathcal{Q}}(vu_i)$ corresponding to this $K_m$
for $vu_{m+1}$ to be adjacent to $vu_1,...,vu_m$. (If we use the
monopolized elements corresponding to all trivial cliques in
$\mathcal{Q}$ for $vu_{m+1}$ to be adjacent to $vu_1,...,vu_m$, then
since there is no triangle in $G$ which contains $v$ and $u_{m+1}$
by  our supposition before, so in any clique partition of $L(G)$ we
can only cover the edge $\{vu_{m+1},vu_{m+2}\}$ by a clique induced
by some star in $G$ centered at $v$. Thus we can use neither the
element in $\bigcup_{i=1}^mS_{\mathcal{Q}}(vu_i)$ corresponding to
$K_m$ nor all monopolized elements in
$\bigcup_{i=1}^mS_{\mathcal{Q}}(vu_i)$ for $vu_{m+2}$, or otherwise
either we can't make  $vu_{m+2}$ be adjacent to $vu_{m+1}$ or we
will make the representation sets of $vu_{m+1},vu_{m+2}$ overlap on
more than one element.) Thus in this case, when $t>m+1$ we have only
one method to make a vertex belonging to $vu_{m+1},...,vu_t$ but not
adjacent to any member of it be   adjacent to $vu_1,...,vu_m$ using
elements in $\bigcup_{i=1}^mS_{\mathcal{Q}}(vu_i)$, while when
$t=m+1$ we have two methods to make $vu_{m+1}$ be adjacent to
$vu_1,...,vu_m$ using elements in
$\bigcup_{i=1}^mS_{\mathcal{Q}}(vu_i)$.

As for the forth case, i.e. one as in Figure 2, for securing the
representation sets of any two vertices from overlapping on more
than one element, we need one pair of vertex-disjoint cliques in the
clique partition as in Figure 2, and the unique two vertex-disjoint
pairs of cliques, refer to Figure 2, are $\{Q_3,Q_4\}$ and
$\{Q_5,Q_6\}$. If we use $Q_3,Q_4$ to make $vu_{m+1},...,vu_t$ be
adjacent to $v_2,v_3,v_4,v_5$, then to make $v_1$ be  adjacent to
$vu_{m+1},...,vu_t$ we can use neither 1 nor 2 for the sake of two
representation sets overlapping on more than one element. Similarly
for the use of $Q_5,Q_6$. Thus in this case we have failed.

For the fifth case, i.e. projective plane, say of order $k\geq 2$,
with one vertex, say $x$, deleted, we know that a clique in this
clique partition has at most $k+1$ vertices, whereas there are
$k^2+k$ vertices in total. Thus in this case each  clique has at
least $$(k^2+k)-(k+1)=k^2-1\geq 3$$ vertices not on it. Thus in
order that $vu_{m+1},...,vu_t$ be adjacent to $vu_1,...,vu_m$, we
need more than one element from $\mathcal{F}(\mathcal{Q})$. Recall
that a projective plane with $k^2+k+1$ points for some $k\geq 2$ has
point and line regularity $k+1$. Thus deleting one vertex from a
projective plane of order $k\geq 2$ leaves a clique partition
consisting of $k+1$ cliques of cardinality $k$ and $k^2$ cliques of
cardinality $k+1$. Besides, recall that any two cliques in a
projective plane intersect on a common vertex. Thus we couldn't
adopt two elements in $\bigcup_{i=1}^mS_{\mathcal{Q}}(vu_i)$ which
correspond to two cliques in $\mathcal{Q}$ of cardinality $k+1$ to
make $vu_{m+1},...,vu_t$ be adjacent to $vu_1,...,vu_m$, or
otherwise the representation set (turned out after we take
$\mathcal{F}(\mathcal{Q})$) of the vertex on which the two cliques
intersect and the representation sets of $vu_{m+1},...,vu_t$ would
overlap on more than one element (unless $t=m$, that is, $G$ is a
star). Nor could we adopt two elements in
$\bigcup_{i=1}^mS_{\mathcal{Q}}(vu_i)$ corresponding to two cliques
in $\mathcal{Q}$ respectively of cardinality $k,k+1$, for the same
reason. Now the only permissible choice is the adoption of elements
in $\bigcup_{i=1}^mS_{\mathcal{Q}}(vu_i)$ corresponding to the $k+1$
cliques in $\mathcal{Q}$ of cardinality $k$. The vertex, say $x$, on
which these $k+1$ cliques would intersect but for the deletion of
$x$ from the primitive projective plane of order $k$, having been
deleted, these $k+1$ cliques are pairwisely vertex-disjoint. (Recall
the property of one linear space that any two lines intersect on at
most one point.) There are altogether $k(k+1)=k^2+k$ vertices in
these $k+1$ cliques, tantamount to the sum total of vertices in
$\mathcal{Q}$. Thus we could utilize the $k+1$ elements
corresponding to these $k+1$ cliques in order that
$vu_{m+1},...,vu_t$ be adjacent to $vu_1,...,vu_m$. Note that this
method can be carried out only when $t=m+1$ for securing two
vertices from having their representation sets overlapping on more
than one element.

To summerize we have obtained the following lemma.

\begin{lem}\label{linegraph anti representation lem}
Let $G$ be a connected simple graph, and $G \neq K_3,K_4,3K_2\vee
K_1, \mbox{or } W_t, t\geq 2$. In addition, we suppose that $G$ is
not a star. And let $u_1^{(i)},...,u_{m_i}^{(i)}$ be all vertices in
$G$ of degree one and adjacent to $v_i$ with $d(v_i)>1$, while
$u_{m_i+1}^{(i)},...,u_{t_i}^{(i)}$ be  all vertices  in $G$ of
degree more  than one and adjacent to $v_i$. We suppose that there
are $k$ vertices with its degree more than one in $G$ in total which
are adjacent to some vertex of degree one, i.e., $1\leq i\leq k$.

Then for any $1\leq i\leq k$ so that $t_i=m_i+1$, we have exactly
four distinct minimum antichain representations of $L(G)$
respectively corresponding to four distinct methods for representing
the clique of $L(G)$ with vertex set
$\{v_iu_1^{(i)},...,v_iu_{t_i}^{(i)}\}$. Figure 7 illustrates these
four distinct methods where for illustration we suppose that $m_i=4$
in the upper three graphs and that $v_iu_1^{(i)},...,v_iu_6^{(i)}$
form a  projective plane of order $2$ with one vertex deleted in the
lowermost graph. Note that the method corresponding to the lowermost
graph of Figure 7 relies on the existence of projective plane with
$t_i$ vertices.
\begin{figure}
\centering
\includegraphics[width=0.8\textwidth]{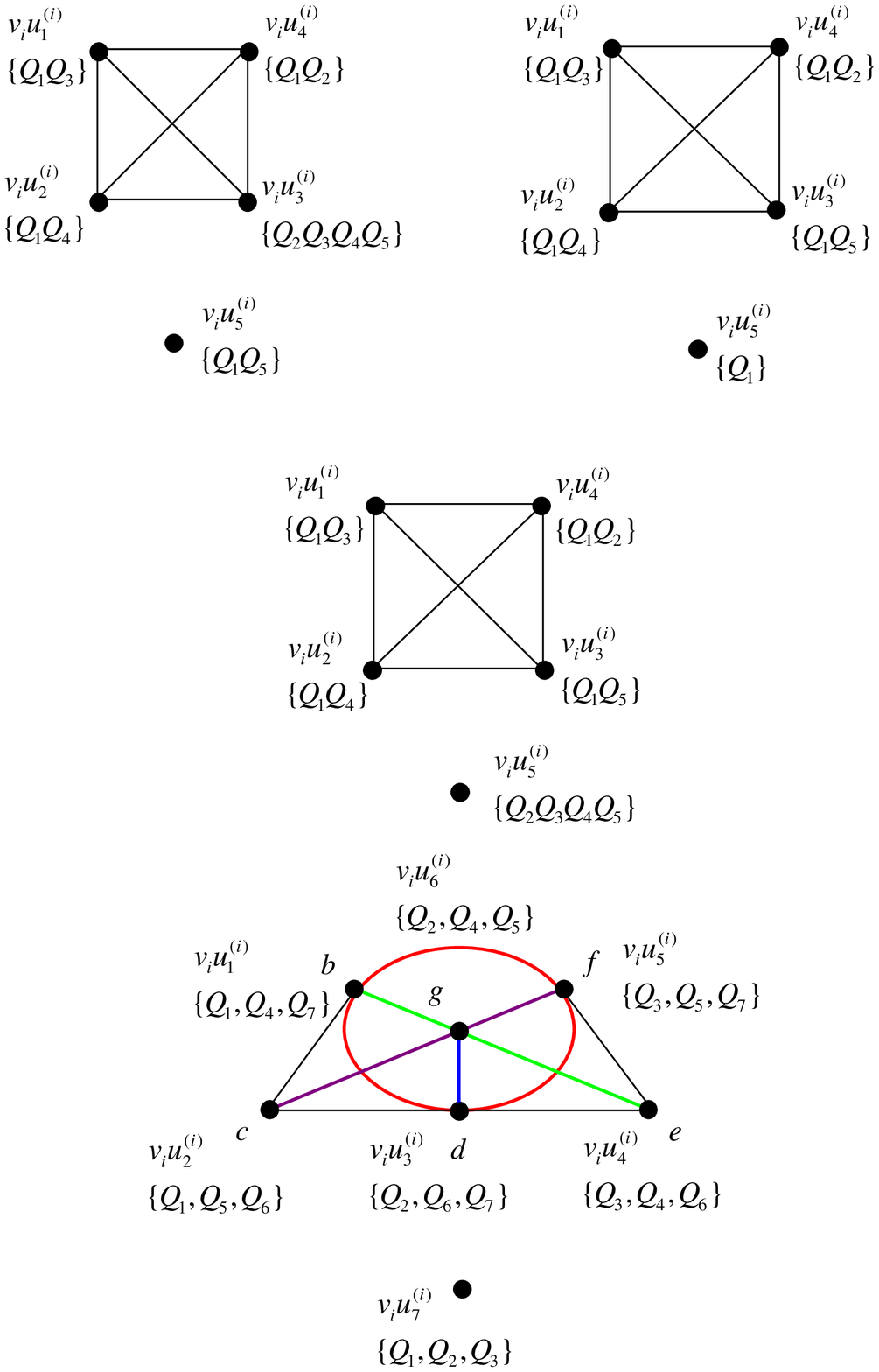}
\caption{Example}
\end{figure}

On the other hand, for any $1\leq i\leq k$ so that $t_i>m_i+1$, all
minimum antichain representations of $L(G)$ have the same method for
representing the clique of $L(G)$ with vertex set
$\{v_iu_1^{(i)},...,v_iu_{m_i}^{(i)}\}$, and for any vertex in
$\{v_iu_{m_i+1}^{(i)},...,v_iu_{t_i}^{(i)}\}$ which is not adjacent
to any other member of it, all minimum antichain representations of
$L(G)$ also have the same  method to make this vertex be adjacent to
$v_iu_1^{(i)},...,v_iu_{m_i}^{(i)}$ using elements in
$\bigcup_{j=1}^{m_i}S(v_iu_j^{(i)})$. Figure 8 illustrate this
unique method, where for illustration we suppose that $m_i=4$ and
$v_iu_{m_i+1}^{(i)}$ is a such vertex, which is not adjacent to any
other member of $\{v_iu_{m_i+1}^{(i)},...,v_iu_{t_i}^{(i)}\}$.
\begin{figure}
\centering
\includegraphics[width=0.6\textwidth]{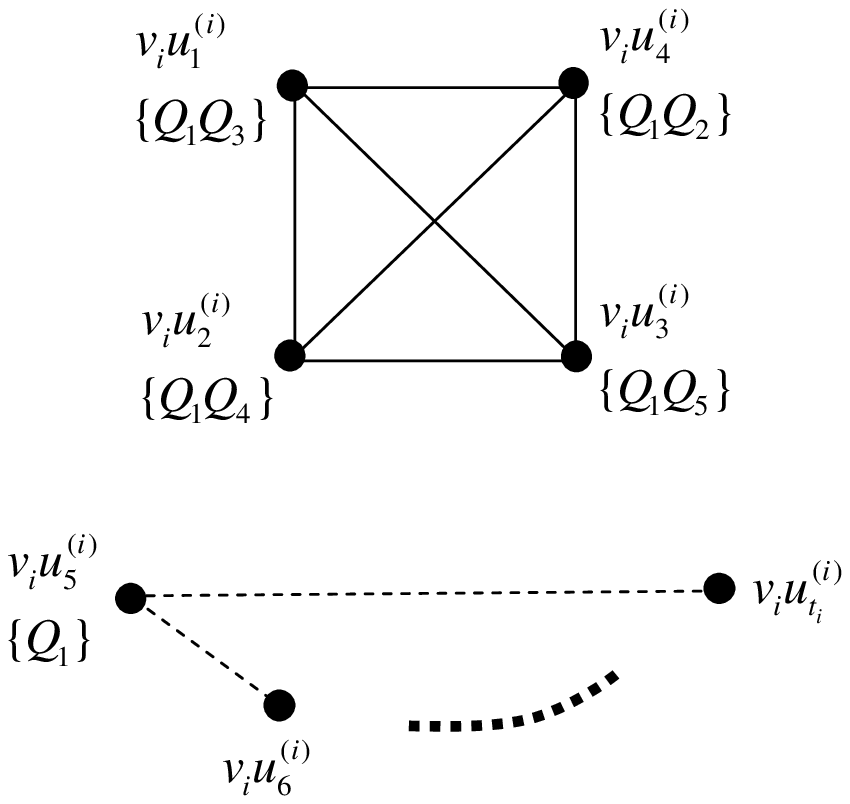}
\caption{Example}
\end{figure}
For any vertex in $\{v_iu_{m_i+1}^{(i)},...,v_iu_{t_i}^{(i)}\}$
which is adjacent to some other member of it, a minimum antichain
representation of $L(G)$ would make this vertex be adjacent to
$v_iu_1^{(i)},...,v_iu_{m_i}^{(i)}$ using elements in
$\bigcup_{j=1}^{m_i}S(v_iu_j^{(i)})$ by one of the two methods as
described in Figure 9, where for illustration we suppose that
$m_i=4$ and $v_iu_{m_i+1}^{(i)}$ is a such vertex, which is adjacent
to $v_iu_{m_i+2}^{(i)}$.

\begin{figure}
\centering
\includegraphics[width=0.9\textwidth]{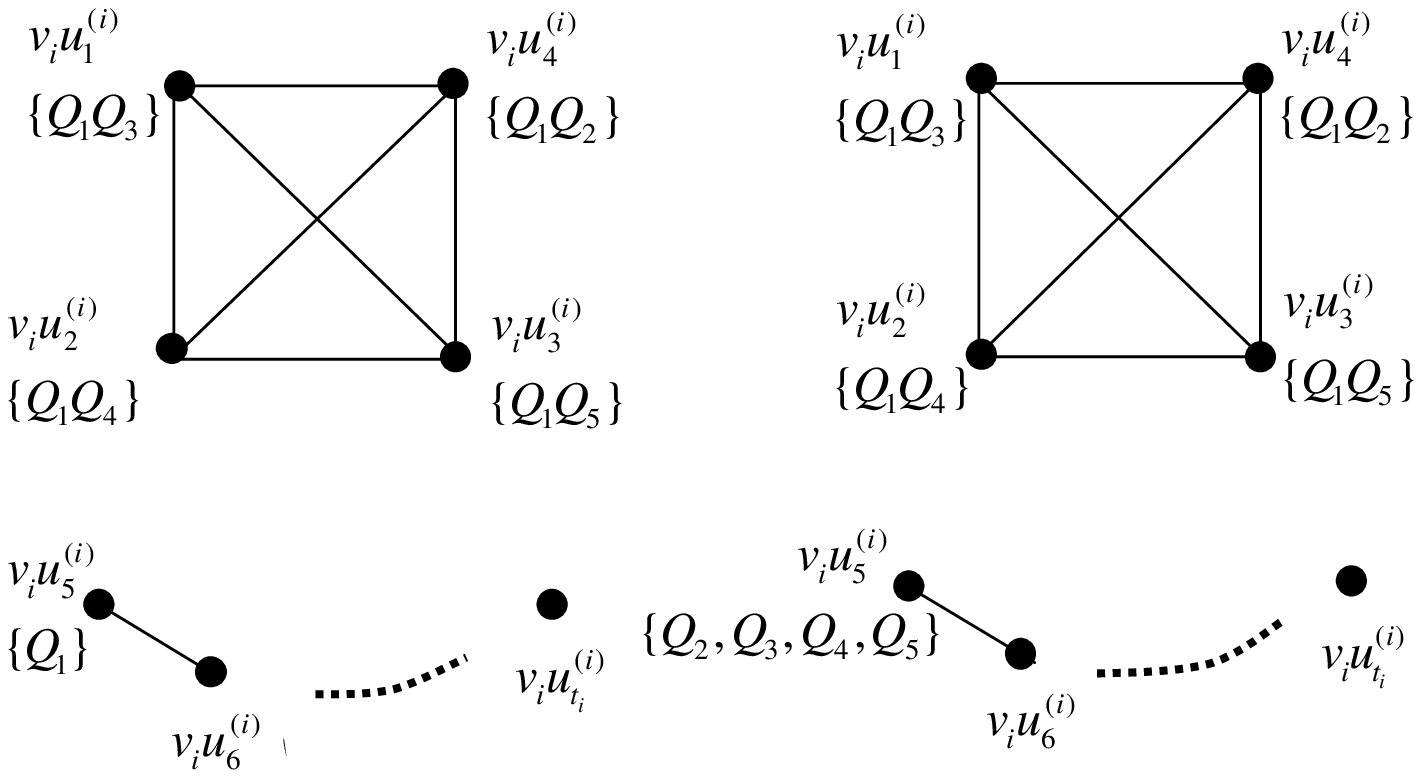}
\caption{Example}
\end{figure}

Note that once a vertex in $\{vu_{m_i+1}^{(i)},...,vu_{t_i}^{(i)}\}$
adopts the representation method as the right in Figure 9, then all
other vertices in $\{vu_{m_i+1}^{(i)},...,vu_{t_i}^{(i)}\}$ must all
adopt the representation method as the left in Figure 9, or
otherwise there will be two vertices in
$\{vu_{m_i+1}^{(i)},...,vu_{t_i}^{(i)}\}$ overlapping on more than
one element on their representation sets.
\end{lem}

Due to the above lemma, what is still vague is mere the case that
$d(u_1)=1$ and there is some triangle on $v$ in $G$, see Figure 10
where for illustration we suppose that $u_1,...,u_m$ are the all
vertices in $G$ adjacent to $v$ and with degree one, $u_{m+1}$ is a
vertex adjacent to $v$ and with degree at least two so that there is
no triangle in $G$ containing the edge $vu_{m+1}$, and
$v,u_{m+2},u_{m+3}$ form a triangle in $G$.

\begin{figure}
\centering
\includegraphics[width=0.6\textwidth]{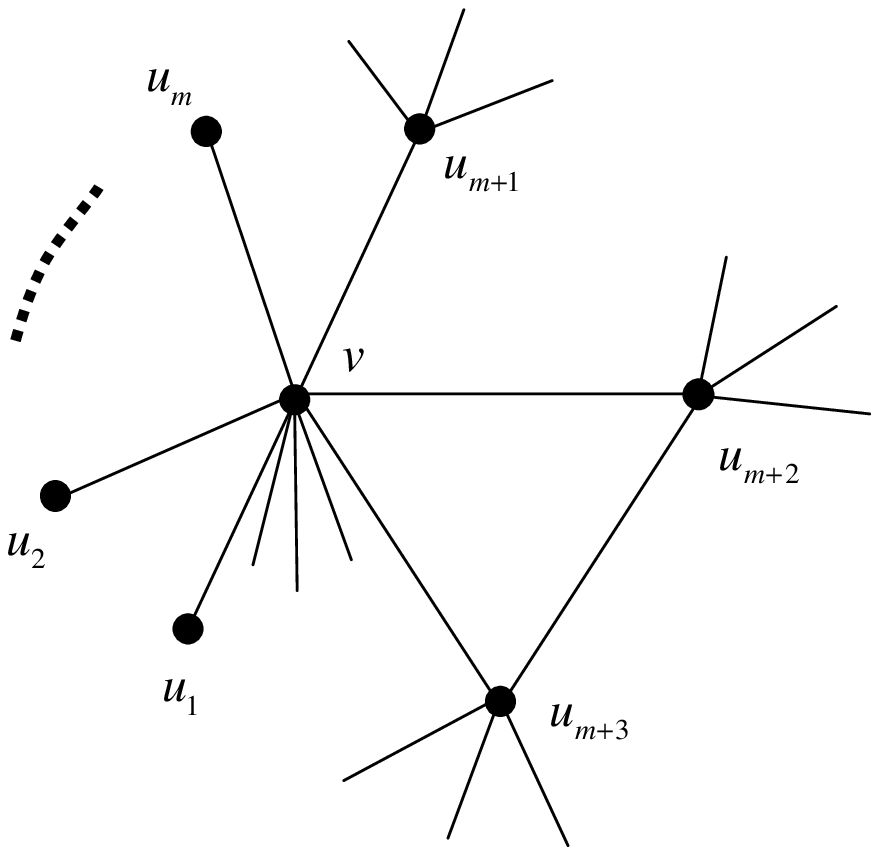}
\caption{Example}
\end{figure}

By Lemma~\ref{linegraph anti representation lem}, if only we can
prove that using the method as the left in Figure 9 is always not
worst than the one as the right in Figure 9 in sense of the intent
to minimize a representation of $L(G)$, where $G$ is connected,
$\neq K_3,K_4,3K_2\vee K_1,\mbox{ or } W_t, t\geq 2$ and is not a
star, and characterize all situations under which the two methods in
Figure 9 is equally fine, then we can determine the antichain
intersection number of any line graph and whether or not any line
graph is uniquely intersectable with respect to antichain.

We examine the method as the left in Figure 9. In this method, refer
to Figure 10, we use one element to make the vertices
$vu_1,...,vu_t$, where we say that $d(v)=t$, be adjacent to each
other, and use $m$ monopolized elements respectively in the
representation sets of $vu_1,...,vu_m$. Thus in the whole $L(G)$, we
use $|V_2(G)|+\sum_{i=1}^km_i$ elements, where $V_2(G)$ denote the
set of vertices of degree at least two in $G$ and we let $v_i,1\leq
i\leq k$ be all vertices of degree more than one in $G$ which is
adjacent to some vertex of degree one and for $1\leq i\leq k$,
$u_1^{(i)},...,u_{m_i}^{(i)}$ be all vertices in $G$ of degree one
and adjacent to $v_i$.

Immediately after we examine the method as the right in Figure 9. In
this method, refer to Figure 10, we use $m$ elements to make
$vu_{m+2}$ be adjacent to $vu_1,...,vu_m$, respectively; and use one
more element to make all $u_i$ with $1\leq i\leq t$ and $i\neq m+2$
be adjacent to each other. Note that now we have made use of $m+1$
elements, that is exactly equal to the number of elements we should
have used for the $v$-star if we had adopted the left method in
Figure 9. But now we should use still another element to make
$vu_{m+2}$ be adjacent to  $vu_{m+1}$ (unless $u_{m+1}$ is adjacent
to $u_{m+2}$ and thus we can shake off the responsibility to make
$vu_{m+2}$ be adjacent to $vu_{m+1}$ to the triangle
$\{vu_{m+1},vu_{m+2},u_{m+1}u_{m+2}\}$ just like how we will deal
with the responsibility to make $vu_{m+2}$ be adjacent to
$vu_{m+3}$). But even if $u_{m+1}$ is adjacent to $u_{m+2}$, where
to dispose of the triangle $\{vu_{m+1},vu_{m+2},u_{m+1}u_{m+2}\}$?
If only $d(u_{m+1})=2$, we can shake off this triangle to the star
$\{vu_{m+1},u_{m+1}u_{m+2}\}$. Thus to attain a minimum antichain
representation we should have either that $u_{m+1}$ is adjacent to
$u_{m+2}$ and $d(u_{m+1})=d(u_{m+3})=2$, or that $d(u_{m+1})=1$ and
$d(u_{m+3})=2$. For the latter case, see Figure 11, where note that
by symmetry we also have all neighbors of $u_{m+2}$ being of degree
one.
\begin{figure}
\centering
\includegraphics[width=0.6\textwidth]{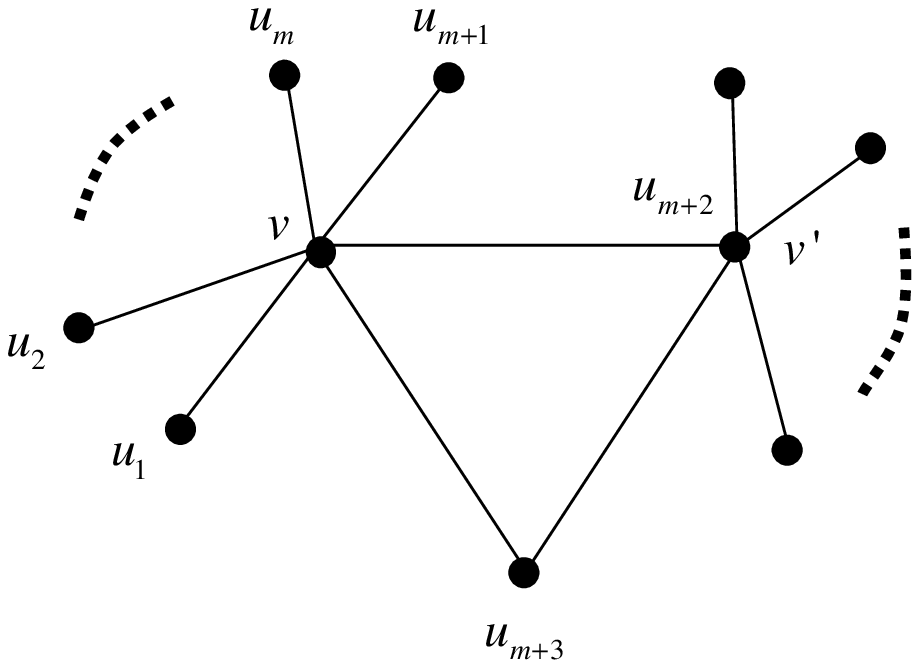}
\caption{Example}
\end{figure}
When $d(u_{m+1})=1$ and  $d(u_{m+2})=d(u_{m+3})=2$, see Figure 12.
\begin{figure}
\centering
\includegraphics[width=0.5\textwidth]{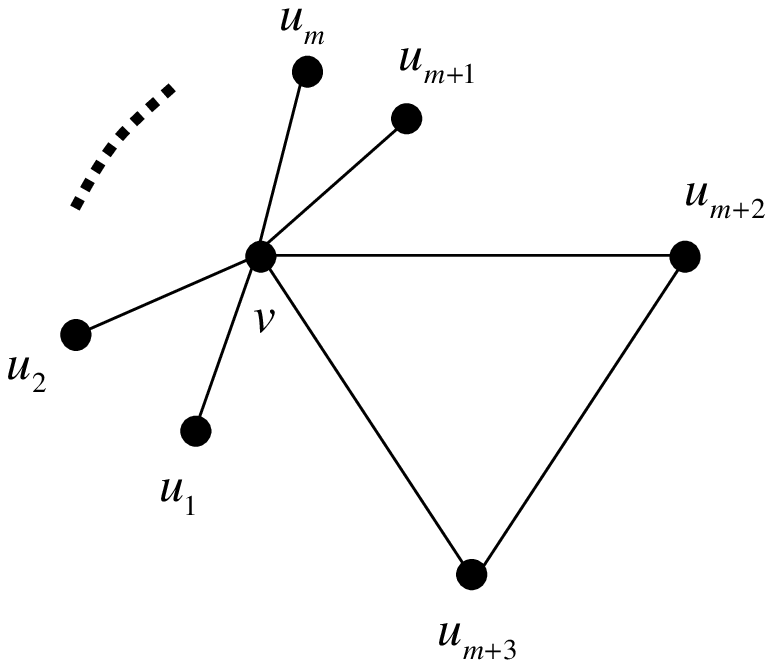}
\caption{Example}
\end{figure}
As for the case that $u_{m+1}$ is adjacent to $u_{m+2}$ and
$d(u_{m+1})=d(u_{m+3})=2$, see Figure 13. There is another case
left, see Figure 14. Therefore we complete the proof.
\end{proof}
\begin{figure}
\centering
\includegraphics[width=0.5\textwidth]{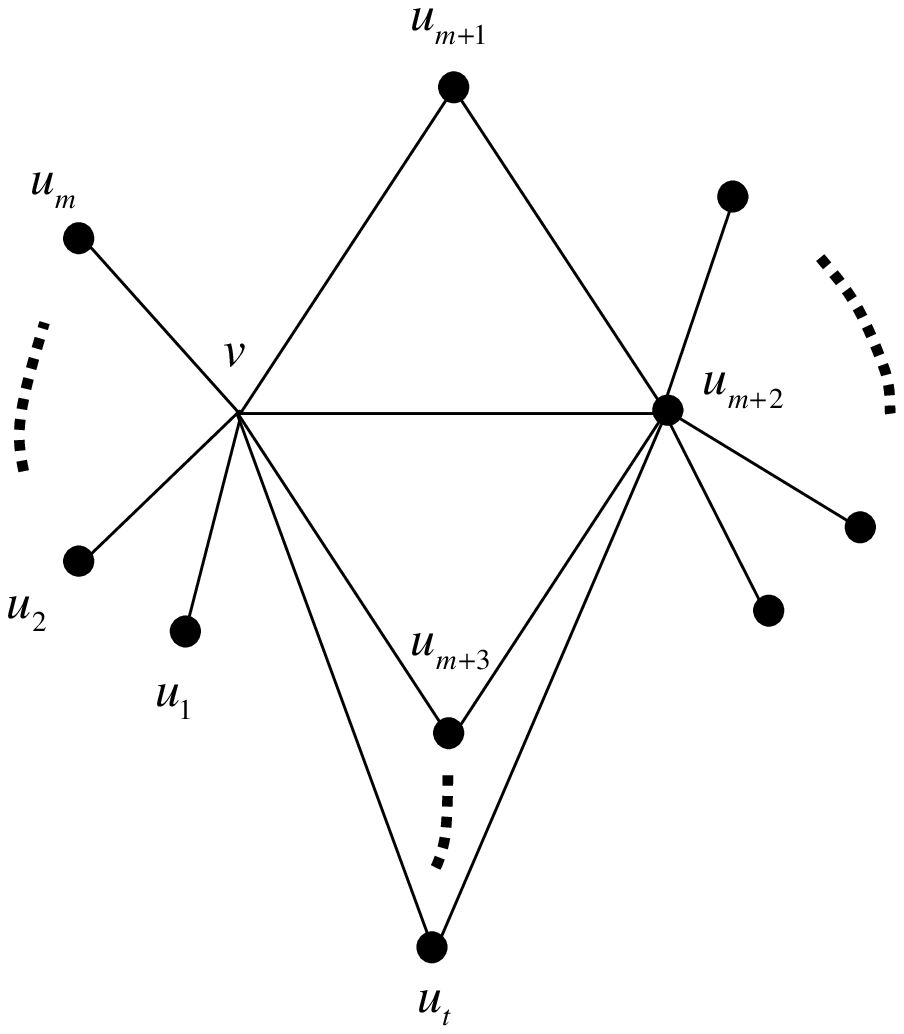}
\caption{Example}
\end{figure}
\begin{figure}
\centering
\includegraphics[width=0.5\textwidth]{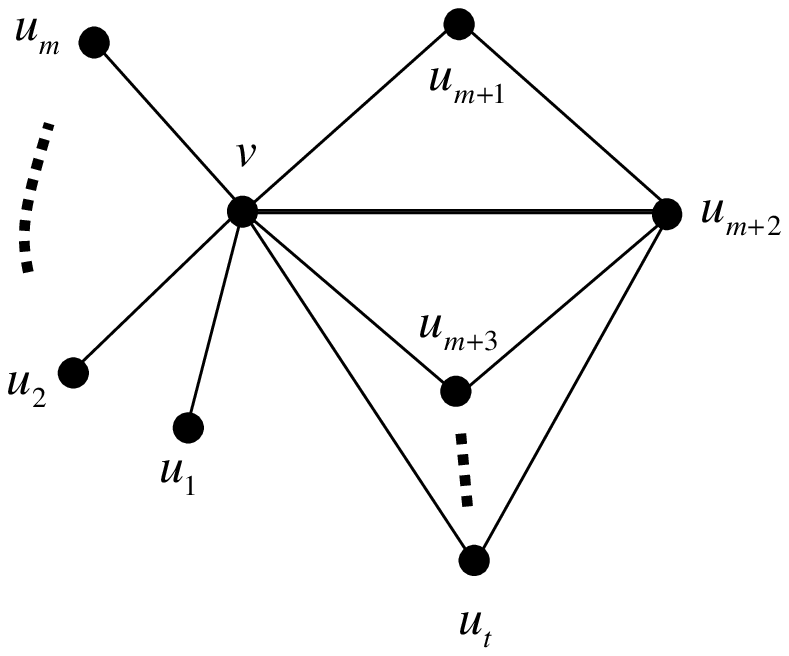}
\caption{Example}
\end{figure}

In Figure 11, if we use the triangle
$\{vu_{m+2},vu_{m+3},u_{m+2}u_{m+3}\}$, i.e., use one element to
make the three vertices $vu_{m+2},vu_{m+3},u_{m+2}u_{m+3}$ be
adjacent to each other, and either use one more element to make all
$vu_i$ with $1\leq i\leq m+3$ and  $i\neq m+2$ be adjacent to each
other and $m+1$ more elements to respectively make $vu_{m+2}$ be
adjacent to $vu_1,...,vu_{m+1}$ or exchange the roles of $u_{m+2}$
and $u_{m+3}$ and do the same as before, and then similarly for the
$v'$-star, then we can obtain four more distinct minimum antichain
representations different from the "canonical one".

In Figure 12, if we use the triangle
$\{vu_{m+2},vu_{m+3},u_{m+2}u_{m+3}\}$, and use one more element to
make all $vu_i$ with $1\leq i\leq m+3,i\neq m+2$ be adjacent to each
other, and use $m+1$ more elements to make $vu_{m+2}$ be adjacent to
$vu_1,...,vu_{m+1}$, respectively, and then attach one monopolized
element to the representation set of $u_{m+2}u_{m+3}$ then we obtain
one more minimum antichain representation other than "the canonical
one".

In Figure 13, if we use the $t-m$ triangles
\begin{align*}
\{vu_{m+1},vu_{m+2},u_{m+1}&u_{m+2}\},\{vu_{m+2},vu_{m+3},u_{m+2}u_{m+3}\},\\
...,\{&vu_{m+2},vu_{t},u_{m+2}u_{t}\},
\end{align*}
and use one more element to make all $vu_i$ with $1\leq i\leq
t,i\neq m+2$ be adjacent to each other, and use $m$ more elements to
make $vu_{m+2}$ be adjacent to $vu_1,...,vu_m$,  respectively, and
then do the same for the $u_{m+2}$-star, we will obtain one more
minimum antichain representation other than ``the canonical one''.

In Figure 14, if we use the $t-m$ triangles
\begin{align*}
\{vu_{m+1},vu_{m+2},u_{m+1}&u_{m+2}\},\{vu_{m+2},vu_{m+3},u_{m+2}u_{m+3}\},\\
...,\{&vu_{m+2},vu_{t},u_{m+2}u_{t}\},
\end{align*}
and use one more element to make all $vu_i$ with $1\leq i\leq
t,i\neq m+2$ be adjacent to each other, and then use one more
element to make all $u_{m+2}u_i$ with $m+1\leq i\leq t$ and $i\neq
m+2$ be adjacent to each other, we will obtain one more minimum
antichain representation other than "the canonical one".

\section{\bf{Conclusion Remarks}}

Edge clique partitions, as a special case of edge clique covers, are
served as great classifying and clustering tools in many practical
applications, therefore it is interesting to explore the concept in
more details.

One may keep working on the set representations of graphs in various
senses. Also the relationships among these various representations
are interesting to be explored further.

\vspace*{0.8cm}
\noindent {\bf{Acknowledgements}}\\

We are very much grateful to the support from National Science
Council of Taiwan, Republic of China, under the grants NSC
96-2115-M-029-001 and NSC 96-2115-M-029-007.


\begin{thebibliography}{99}


\bibitem{alter}R. Alter and C. C. Wang, Uniquely intersectable graphs, \textit{Discrete Mathematics}
\textbf{18} (1977) 217-226.


\bibitem{batten}L. M. Batten, \textit{Combinatorics of Finite Geometries,}
Cambridge University Press, Cambridge, New York, Melbourne (1986).


\bibitem{batten}L. M. Batten and A. Beutelspacher, \textit{The theory of finite linear spaces,}
Cambridge University Press (1993).


\bibitem{bridges}W. G. Bridges, Near 1-designs, \textit{Journal of Combinatorial Theory (Series A)}
\textbf{Volume 13 Issue 1} (July 1972) 116-126.


\bibitem{bruijn}N. G. de Bruijn and P. Erd\"{o}s  (1948), On a combinatorial problem, \textit{Indag. Math.}
\textbf{10} 421-423 and \textit{Nederl. Akad. Wetensch. Proc. Sect.
Sci.} \textbf{51} 1277-1279.


\bibitem{bylka}S. Bylka and J. Komar, Intersection properties of line graphs, \textit{Discrete Mathematics}
\textbf{164} (1997) 33-45.


\bibitem{tsuchiya2}H. Era and M. Tsuchiya, On intersection graphs with respect to uniform families, \textit{Utilitas Math.}
\textbf{37} (1990) 3-12.



\bibitem{erdos}P. Erd\"{o}s, A. Goodman, and L. P\'{o}sa, The representation of a graph by set intersections,
\textit{Canad. J. Math.} \textbf{18} (1966) 106-112.


\bibitem{wang}N. V. R. Mahadev and T.-M. Wang, On uniquely intersectable graphs, \textit{Discrete Mathematics}
\textbf{207} (1999) 149-159.


\bibitem{mcguinness2}S. McGuinness and R. Rees, On the number of distinct minimal clique partitions and
clique covers of a line graph, \textit{Discrete Math.} \textbf{83} (1990) 49-62.


\bibitem{lam} C. W. H. Lam, The Search for a Finite Projective Plane of Order 10,
\textit{Amer. Math. Monthly} \textbf{98}, (1991) 305-318.


\bibitem{orlin}J. Orlin, Contentment in graph theory: covering graphs with cliques, \textit{Indag. Math.}
\textbf{39} (1977) 406-424.

\bibitem{ESM}E. Szpilrajn-Marczewski, Sur deux proprietes des classes d'ensembles, \textit{Fundamenta Mathematicae}
\textbf{33} (1945) 303-307.


\bibitem{tsuchiya}M. Tsuchiya, On intersection graphs with respect to antichains (II),
\textit{Utilitas Math.} \textbf{37} (1990) 29-44.

\bibitem{wang}Tao-Ming Wang, \textit{On Uniquely Intersectable Graphs}, Ph.D. Thesis, Northeastern University, Boston, MA. USA, 1997.

\bibitem{west}D. B. West, \textit{Introduction to Graph Theory}, Second Edition, Prentice-Hall,
Upper Saddle River, NJ, 2004.


\end{thebibliography}
\end{document}